\documentclass[11pt]{amsart}

\usepackage{amsmath, amsfonts, amssymb,amsthm}
\usepackage{amstext}
\usepackage{mathrsfs}

\usepackage{float,epsf,subfigure}
\usepackage[all,cmtip]{xy}
\usepackage{hyperref}
\usepackage{algorithm}
\usepackage{algorithmic}
\usepackage{mathtools}
\usepackage{float}
\usepackage{bm}
\theoremstyle{plain}
\newtheorem{theorem}{Theorem}[section]

\newtheorem{cor}[theorem]{Corollary}
\newtheorem{def-thm}[theorem]{Definition-Theorem}
\newtheorem{lemma}[theorem]{Lemma}

\newtheorem*{tha}{Theorem A}

\theoremstyle{definition}

\newtheorem{remark}[theorem]{Remark}

\def\min{\mathop{\mathrm{min}}}

\begin{document}
\title[Nevanlinna theory on complete K\"ahler manifolds]{Nevanlinna theory on complete K\"ahler manifolds}
\author[X.J. Dong]
{Xianjing Dong}

\address{Academy of Mathematics and Systems Science\\ Chinese Academy of Sciences \\ Beijing, 100190, P.R. China}
\email{xjdong@amss.ac.cn}

\thanks{The work was supported by  Postdoctoral Science Fund of China (No. 2020M670488).}

\subjclass[2010]{30D35, 32H30.} \keywords{Nevanlinna theory; Second Main Theorem; K\"ahler manifold; Ricci curvature; Brownian motion.}
\date{}
\maketitle \thispagestyle{empty} \setcounter{page}{1}

\begin{abstract}
\noindent We study  Nevanlinna theory on complete  K\"ahler manifolds. 
 As a consequence of the main result, we prove a defect relation of holomorphic mappings from  complete K\"ahler manifolds of non-positive sectional curvature into complex projective manifolds under certain  growth condition. 
\end{abstract}

\vskip\baselineskip

\setlength\arraycolsep{2pt}
\medskip

\section{Introduction}

Early in 1972, Carlson and Griffiths  \cite{gri, gth} established the equi-distribution theory of  holomorphic mappings from
$\mathbb C^m$ into complex projective  manifolds intersecting divisors.
Later,  Griffiths and King \cite{gri1,gth}  further generalized  the theory from $\mathbb C^m$ to  affine  manifolds. 
Let us  review  this  theory briefly.

Let $V$ be a complex projective  manifold of complex dimension $ m,$ and
 let $L\rightarrow V$ be a positive line bundle over $V.$ For a reduced divisor $D$  on $V$ and a  holomorphic mapping $f:\mathbb C^m\rightarrow V,$
 we have the standard  notations $T_f(r,L),$ $m_f(r,D)$ and $N_f(r,D)$ in Nevanlinna theory (see \cite{Noguchi,ru} or Remark \ref{remark}).
Carlson and Griffiths proved the following Second Main Theorem
\begin{tha}\label{sd} Let $L\rightarrow V$ be a positive  line bundle and let a reduced divisor  $D\in|L|$  be of simple normal
crossing type.
Let $f:\mathbb C^m\rightarrow V$ be a differentiably non-degenerate equi-dimensional holomorphic mapping. Then for any $\delta>0$
  \begin{eqnarray*}
     T_f(r,L)+T_f(r,K_V)
     \leq N_f(r,D)+O\big{(}\log T_f(r,L)+\delta\log r\big{)}
  \end{eqnarray*}
  holds for all $r>1$ outside a set $E_\delta\subset(1,\infty)$ of finite Lebesgue measure.
\end{tha}

Theorem A was extended by Sakai \cite{Sakai} in terms of Kodaira dimension,
and generalized by Shiffman \cite{Shiff} in the singular divisor case. More  investigations
were done by Wong, Lang, Cherry and Noguchi (see \cite{gri1,Lang,Noguchi,No,ru,wong}).

The purpose of this paper is to generalize  Theorem A to complete  K\"ahler manifolds.
 Our approach is to combine stochastic method with Wong-Lang's technique.
Recall that the Brownian motion method was first used by Carne \cite{carne} in proving Nevanlinna's Second Main Theorem of meromorphic functions on $\mathbb C.$  Later, Atsuji \cite{at, at1, at2, atsuji} developed this technique to study the Second Main Theorem of meromorphic functions on
 complete K\"ahler manifolds.
 Recently, Dong-He-Ru \cite{Dong} re-visited this technique and
    provided a  probabilistic proof of Cartan's Second Main Theorem of holomorphic curves.

We  state the main result. For  technical reasons, all  manifolds (as domains)  are assumed to be open in this paper.
  Let $M$ be a  complete K\"ahler manifold of non-positive sectional curvature with complex dimension $m=\dim_{\mathbb C}V.$
  For a holomorphic mapping $f:M\rightarrow V,$
one can extend  the definition of classical Nevanlinna's functions (see Section 2) to K\"ahler manifold $M$ naturally.
  Let ${\rm{Ric}}_M$  be the Ricci curvature tensor of $M,$
 set
\begin{equation}\label{kappa}
  \kappa(t)=\frac{1}{2m-1}\min_{x\in \overline{B_o(t)}}R_M(x),
\end{equation}
where $R_M(x)$ is the
pointwise lower bound of Ricci curvature defined by
$$ R_M(x)=\inf_{\xi\in T_{x}M, \ \|\xi\|=1}{\rm{Ric}}_M(\xi,\overline{\xi}).$$
 \begin{theorem}\label{thm1} Let $L\rightarrow V$ be a positive  line bundle and let a reduced divisor $D\in|L|$ be of simple normal
crossing type. Let  $f:M\rightarrow V$ be a differentiably non-degenerate equi-dimensional holomorphic mapping. Then
  \begin{eqnarray*}
     & & T_f(r,L)+T_f(r,K_V)-N^{[1]}_f(r,D) \\
    &\leq& \frac{m+k}{2}\log T_f(r,L)+O\Big(\log^+\log T_f(r,L)-\kappa(r)r^2+\log^+\log r\Big)
  \end{eqnarray*}
  holds for all $r>1$ outside a set of finite Lebesgue measure, where  $k$ is the complexity of $D$ defined by $(\ref{com}).$
\end{theorem}
The term $\kappa(r)$ appeared in Theorem \ref{thm1} depends  on the curvature of $M$.  Consider the simple case where $M=\mathbb C^m,$  we have $\kappa(r)\equiv0$ and $T_f(r,L)\geq O(\log r)$ as $r\rightarrow\infty.$  It yields from Theorem \ref{thm1}  that
  \begin{eqnarray*}
     T_f(r,L)+T_f(r,K_V)-N^{[1]}_f(r,D)
    \leq \frac{m+k}{2}\log T_f(r,L)+\text{\emph{Lower order terms}}.
  \end{eqnarray*}
So, Theorem \ref{thm1} implies Theorem A.
Coefficient $(m+k)/2$ before $\log T_f(r,L)$ is optimal. When $m=1,$ we have $k=1$ and $(m+k)/2=1.$ It is mentioned that Ye \cite{Ye} showed  the estimate ``1" is best.

As a consequence of Theorem \ref{thm1}, we  derive a defect relation.
Recall that
the  \emph{defect} (without counting multiplicities) of $f$   with respect to $D$ is defined by
$$\Theta_f(D):=1-\limsup_{r\rightarrow\infty}\frac{N^{[1]}_f(r,D)}{T_f(r,L)}. $$
In general, we set for two  holomorphic line bundles $L_1, L_2$ over $V$ that
$$\overline{\left[\frac{c_1(L_2)}{c_1(L_1)}\right]}:=\inf\left\{s\in\mathbb R: \omega_2<s\omega_1, \ ^\exists \omega_1\in c_1(L_1),
\ ^\exists\omega_2\in c_1(L_2) \right\}.$$

\begin{cor}\label{defect} Assume the same conditions as in Theorem $\ref{thm1}.$
 If $f$ satisfies the growth condition
$$ \liminf_{r\rightarrow\infty}\frac{\kappa(r)r^2}{T_f(r,L)}=0,$$
then
$$\Theta_f(D)\leq
\overline{\left[\frac{c_1(K^*_V)}{c_1(L)}\right]}.$$
\end{cor}
In particular, when $M=\mathbb C^m,$ we  derive Carlson-Griffiths' defect relation.
\section{Basic notations}
\subsection{Brownian motions}~

We first introduce Brownian motions in Riemannian manifolds \cite{13,NN,itoo}  and notions of Nevanlinna's functions,  then we give the First Main Theorem of Nevanlinna theory.

 Let  $(M,g)$ be a Riemannian manifold with  Laplace-Beltrami operator $\Delta_M$ associated to  $g.$  For $x\in M,$ we denote by $B_x(r)$ the geodesic ball centered at $x$ with radius $r,$ and denote by $S_x(r)$ the geodesic sphere centered at $x$ with radius $r.$
 By Sard's theorem, $S_x(r)$ is a submanifold of $M$ for almost every $r>0.$
A Brownian motion $X_t$ in $M$
is a heat diffusion  process  generated by $\frac{1}{2}\Delta_M$ with transition density function $p(t,x,y)$ which is the minimal positive fundamental solution of the  heat equation
  $$\frac{\partial}{\partial t}u(t,x)-\frac{1}{2}\Delta_{M}u(t,x)=0.$$
We denote by $\mathbb P_x$ the law of $X_t$ started at $x\in M$
 and by $\mathbb E_x$ the corresponding expectation with respect to $\mathbb P_x.$

 \noindent\textbf{Co-area formula and Dynkin formula}

  Let $D$  be a bounded domain with   smooth boundary $\partial D$ in $M$.
Fix $x\in D,$  we use $d\pi^{\partial D}_x$ to denote the harmonic measure  on $\partial D$ with respect to $x.$
This measure is a probability measure.
 Set
$$\tau_D:=\inf\big\{t>0:X_t\not\in D\big\}$$
which is a stopping time.
Let $g_D(x,y)$  denote the Green function of $\Delta_M/2$ for  $D$ with a pole at $x$ and Dirichlet boundary condition, namely
$$-\frac{1}{2}\Delta_{M}g_D(x,y)=\delta_x(y), \ y\in D; \ \ g_D(x,y)=0, \ y\in \partial D,$$
where $\delta_x$ is the Dirac function.
 For $\phi\in \mathscr{C}_{\flat}(D)$
 (space of bounded continuous functions on $D$), \emph{co-area formula} \cite{bass} asserts  that
$$ \mathbb{E}_x\left[\int_0^{\tau_D}\phi(X_t)dt\right]=\int_{D}g_{D}(x,y)\phi(y)dV(y).
$$
From Proposition 2.8 in \cite{bass}, we also have the relation of harmonic measures and hitting times that
\begin{equation}\label{hello}
  \mathbb{E}_x\left[\psi(X_{\tau_{D}})\right]=\int_{\partial D}\psi(y)d\pi_x^{\partial D}(y)
\end{equation}
for any $\psi\in\mathscr{C}(\overline{D})$.
Since  the expectation $``\mathbb E_x",$ co-area formula and (\ref{hello}) still work in the case when $\phi$ or $\psi$ has a pluripolar set of singularities.

Let $u\in\mathscr{C}_\flat^2(M)$ (space of bounded $\mathscr{C}^2$-class functions on $M$), we have the famous \emph{It\^o formula} (see \cite{at,13, NN,itoo})
$$u(X_t)-u(x)=B\left(\int_0^t\|\nabla_Mu\|^2(X_s)ds\right)+\frac{1}{2}\int_0^t\Delta_Mu(X_s)dt, \ \ \mathbb P_x-a.s.$$
where $B_t$ is the standard  Brownian motion in $\mathbb R$ and $\nabla_M$ is  gradient operator on $M$.
Take expectation of both sides of the above formula, it  follows \emph{Dynkin formula} (see \cite{at, itoo})
$$ \mathbb E_x[u(X_T)]-u(x)=\frac{1}{2}\mathbb E_x\left[\int_0^T\Delta_Mu(X_t)dt\right]
$$
for a stopping time $T$ such that each term  makes sense.
Noting that  Dynkin formula still holds for  $u\in\mathscr{C}^2(M)$ if $T=\tau_D.$
In further, it also works when $u$ is of a pluripolar set of singularities, particularly, for a plurisubharmonic function $u.$ 
\subsection{Nevanlinna's functions}~

Let $$f: M\rightarrow V$$
 be a holomorphic mapping into a compact complex manifold $V.$
Fix  $o\in M$ as a reference point and denote by $g_r(o,x)$ the Green function of $\Delta_M/2$ for  geodesic ball $B_o(r)$ with a pole at $o$ and   Dirichlet boundary condition.
 For  a (1,1)-form $\varphi$  on $M,$ we use the following convenient notations
 $$e_{\varphi}(x):=2m\frac{\varphi \wedge\alpha^{m-1}}{\alpha^m}, \ \ \ T(r,\varphi):=\frac{1}{2}\int_{B_o(r)}g_r(o,x)e_{\varphi}(x)dV(x),$$
 where $dV$ is the Riemannian volume measure of $M.$
For a (1,1)-form $\omega$  on $N,$ 
the
  \emph{characteristic function} of $f$ with respect to $\omega$ is defined  by
$$ T_f(r,\omega):=T(r,f^*\omega).$$
Let $L\rightarrow V$ be a  holomorphic line bundle equipped with  Hermitian metric $h,$
the associated Chern form of $L$ is $c_1(L,h):=-dd^c\log h.$
We define $$T_f(r, L): =T_f(r, c_1(L, h))$$
up to a bounded term.  A simple computation shows that
$$ e_{f^*c_1(L,h)}
=2m\frac{f^*c_1(L,h)\wedge\alpha^{m-1}}{\alpha^m}=-\frac{1}{2}\Delta_M\log(h\circ f).$$
Set
 $$\tau_r:=\inf\big{\{}t>0: X_t\not\in B_o(r)\big{\}},$$
where $X_t$ is the Brownian motion in $M$ generated by $\Delta_M/2$ started at $o.$
By co-area formula, we have
$$T_f(r,L) =\frac{1}{2}\mathbb E_o\left[\int_0^{\tau_r}e_{f^*c_1(L,h)}(X_t)dt\right].$$
Let $\mathscr R_M:=-dd^c\log\det(g_{i\bar{j}})$ be the Ricci curvature form of $(M,g).$
We define the Ricci curvature term  by
\begin{eqnarray*}
   T(r,\mathscr{R}_M)&:=&\frac{1}{2}\int_{B_o(r)}g_r(o,x)e_{\mathscr R_M}(x)dV(x) \\
&=& m\mathbb E_o\left[\int_0^{\tau_r}\frac{\mathscr R_M\wedge\alpha^{m-1}}{\alpha^m}(X_t)dt\right] \\
&=& -\frac{1}{4}\mathbb E_o\left[\int_0^{\tau_r}\Delta_M\log\det\big{(}g_{i\bar{j}}(X_t)\big{)}dt\right].
\end{eqnarray*}
Given $D\in|L|,$ an effective divisor such that $s_D\in H^0(N,L),$  where $s_D$ is the
canonical section defined by $D.$
Since $V$ is compact,  assume that $\|s_D\|<1.$
The \emph{proximity function} of $f$ with respect to $D$ is  defined by
$$  m_f(r,D):=\int_{S_o(r)}\log\frac{1}{\|s_D\circ f(x)\|}d\pi_o^r(x),
$$
where $d\pi_o^r$ is the harmonic measure on  $S_o(r)$ with respect to $o.$
The relation between  harmonic measure and hitting time implies that
$$  m_f(r,D) =\mathbb E_o\left[\log\frac{1}{\|s_D\circ f(X_{\tau_r})\|}\right].$$
 The \emph{counting function} of $f$ with respect to $D$ is  defined by
$$N_f(r,D):=\frac{\pi^m}{(m-1)!}\int_{f^*D\cap B_o(r)}g_r(o,x)\alpha^{m-1},$$
where $$\alpha:=\frac{\sqrt{-1}}{2\pi}\sum_{i,j=1}^m g_{i\bar j}dz_i\wedge d\bar z_j$$ is the K\"ahler metric form of $M$ associated to $g.$
Writing $s_D=\tilde{s}_De_\alpha$ locally, where $\{e_\alpha,U_\alpha\}$ is a local holomorphic frame of $(L,h)$ restricted to $U_\alpha.$ Then we have
\begin{eqnarray*}
   N_f(r,D)&=&\frac{\pi^m}{(m-1)!}\int_{B_o(r)}g_r(o,x)dd^c\log|\tilde{s}_D\circ f|^2\wedge\alpha^{m-1} \\
&=&\frac{1}{4}\int_{B_o(r)}g_r(o,x)\Delta_M\log|\tilde{s}_D\circ f(x)|^2dV(x).
\end{eqnarray*}
\begin{remark}\label{remark} The definitions of Nevanlinna's functions in above are  natural extensions of the classical ones. To see that, we recall the  $\mathbb C^m$ case:
\begin{eqnarray*}
   T_f(r,L)&=&\int_0^r\frac{dt}{t^{2m-1}}\int_{B_o(t)}f^*c_1(L,h)\wedge\alpha^{m-1}, \\
  m_f(r,D)&=&\int_{S_o(r)}\log\frac{1}{\|s_D\circ f\|}\gamma, \\
 N_f(r,D)&=&\int_0^r\frac{dt}{t^{2m-1}}\int_{B_o(t)}dd^c\log|\tilde{s}_D\circ f|^2\wedge\alpha^{m-1},
\end{eqnarray*}
where $o=(0,\cdots,0)$ and
$$\alpha=dd^c\|z\|^2,\ \ \ \gamma=d^c\log\|z\|^2\wedge \left(dd^c\log\|z\|^2\right)^{m-1}.$$
Note  the facts that
$$\gamma=d\pi_o^r(z),  \ \ \ g_r(o,z)=\left\{
                \begin{array}{ll}
                  \frac{\|z\|^{2-2m}-r^{2-2m}}{(m-1)\omega_{2m-1}}, & m\geq2; \\
                  \frac{1}{\pi}\log\frac{r}{|z|}, & m=1.
                \end{array}
              \right.,$$
where $\omega_{2m-1}$ is the volume of  unit sphere in $\mathbb R^{2m}.$ Apply integration by part, we see that the above expressions agree with  ours.
\end{remark}
The Dynkin formula implies that
\begin{theorem}[FMT]\label{first}  Assume that $f(o)\not\in{\rm{Supp}}D.$ Then
$$T_f(r,L)=m_f(r,D)+N_f(r,D)+O(1).$$
\end{theorem}
\section{Calculus Lemma}

Let $M$ be a  simply-connected complete K\"ahler manifold with non-positive sectional curvature.
Let $\kappa$ be defined by (\ref{kappa}), then $\kappa$ is a non-positive, non-increasing and continuous function  on $[0,\infty).$
   Consider the  ODE
 \begin{equation}\label{G}
G''(t)+\kappa(t)G(t)=0;\ \ \ G(0)=0, \ \ G'(0)=1
 \end{equation}
on $[0,\infty).$  Comparing (\ref{G})  with $y''(t)+\kappa(0)y(t)=0$ provided with the same  initial conditions, we see that
 $G$ can be  estimated simply as
$$G(t)=t \ \ \text{for}  \ \kappa\equiv0; \ \ \ G(t)\geq t \ \ \text{for} \ \kappa\not\equiv0.$$
This follows that
\begin{equation}\label{v1}
  G(r)\geq r \ \ \text{for} \ r\geq0; \ \ \ \int_1^r\frac{dt}{G^{2m-1}(t)}\leq\log r \ \ \text{for} \ r\geq1.
\end{equation}
On the other hand, we  rewrite (\ref{G}) in the form
$$\log'G(t)\cdot\log'G'(t)=-\kappa(t).$$
Since $G(t)\geq t$ is increasing,
then the decrease and non-positivity of $\kappa$ imply that for each fixed $t,$ $G$  must be satisfied one of the following two inequalities
$$\log'G(t)\leq\sqrt{-\kappa(t)} \ \ \text{for} \ t>0; \ \ \ \log'G'(t)\leq\sqrt{-\kappa(t)} \ \ \text{for} \ t\geq0.$$
By virtue of $G(t)\rightarrow0$ as $t\rightarrow0,$ by integration, $G$ is bounded from above by
\begin{equation}\label{v2}
  G(r)\leq r\exp\big(r\sqrt{-\kappa(r)}\big) \ \  \text{for} \ r\geq0.
\end{equation}

Before giving the Calculus Lemma, we  introduce some lemmas.
\begin{lemma}[\cite{atsuji}]\label{zz} Let $G(t)$ be defined in {\rm{(\ref{G})}}, and let $\eta>0$ be a constant. Then there exists  a constant $C>0$ such that for
$r>\eta$ and $x\in B_o(r)\setminus \overline{B_o(\eta)},$ we have
  $$g_r(o,x)\int_{\eta}^rG^{1-2m}(t)dt\geq C\int_{r(x)}^rG^{1-2m}(t)dt.$$
\end{lemma}
\begin{lemma}[\cite{Deb}]\label{sing1} We have
$$d\pi^r_{o}(x)\leq\frac{1}{\omega_{2m-1}r^{2m-1}}d\sigma_r(x),$$
 where $d\pi_{o}^r(x)$ is the harmonic measure on geodesic sphere $S_o(r)$ with respect to $o\in M,$  $d\sigma_r(x)$ is the induced volume measure on $S_o(r)$ and $\omega_{2m-1}$ is the Euclidean volume of unit sphere in $\mathbb R^{2m}.$
\end{lemma}

\begin{lemma}[Borel Lemma, \cite{ru}]\label{cal1} Let $T$ be a strictly positive nondecreasing  function of $\mathscr{C}^1$-class on $(0,\infty).$ Let $\gamma>0$ be a number such that $T(\gamma)\geq e,$ and $\phi$ be a strictly positive nondecreasing function such that
$$c_\phi=\int_e^\infty\frac{1}{t\phi(t)}dt<\infty.$$
Then, the inequality
  $T'(r)\leq T(r)\phi(T(r))$
holds for all $r\geq\gamma$ outside a set of Lebesgue measure not exceeding $c_\phi.$ In particular, if take $\phi(t)=\log^{1+\delta}t$ for  $\delta>0,$ then 
  $$T'(r)\leq T(r)\log^{1+\delta}T(r)$$
holds for all $r>0$ outside a set $E_\delta\subset(0,\infty)$ of finite Lebesgue measure.
\end{lemma}

We are ready to prove the following so-called Calculus Lemma
\begin{theorem}[Calculus Lemma]\label{cal} Let $\Gamma\geq0$ be a locally integrable  function on $M$ such that it is locally bounded at $o\in M.$
 Then for any $\delta>0,$ there exists a constant $C>0$ independent of $\Gamma,\delta,$ and a set $E_\delta\subset(1,\infty)$ of finite Lebesgue measure such that
\begin{eqnarray*}
\mathbb E_o\big{[}\Gamma(X_{\tau_r})\big{]}
&\leq& \frac{F(\hat{\Gamma},\kappa,\delta)e^{(2m-1)r\sqrt{-\kappa(r)}}\log r}{C\omega_{2m-1}}\mathbb E_o\left[\int_0^{\tau_{r}}\Gamma(X_{t})dt\right]
\end{eqnarray*}
  holds for $r>1$ outside $E_\delta,$  where $\kappa$ is defined by $(\ref{kappa}),$ $\omega_{2m-1}$ is the Euclidean volume of unit sphere in $\mathbb R^{2m}$ and $F$ is defined by
$$F(\hat{\Gamma},\kappa,\delta) =\Big\{\log^+\hat{\Gamma}(r)\cdot\log^+\Big(r^{2m-1}e^{(2m-1)r\sqrt{-\kappa(r)}}\hat{\Gamma}(r)\big(\log^{+}\hat{\Gamma}(r)\big)^{1+\delta}\Big)\Big\}^{1+\delta}
$$
with
$$\hat{\Gamma}(r)=\frac{\log r}{C}\mathbb E_o\left[\int_0^{\tau_{r}}\Gamma(X_{t})dt\right].$$
Moreover, we have the estimate
$$\log F(\hat{\Gamma},\kappa,\delta)
\leq O\Big(\log^+\log \mathbb E_o\left[\int_0^{\tau_{r}}\Gamma(X_{t})dt\right]+\log^+\big(r\sqrt{-\kappa(r)}\big)+\log^+\log r\Big).
$$
\end{theorem}
\begin{proof}
Combining Lemma \ref{zz} with Lemma \ref{sing1} and (\ref{v1}), we obtain 
\begin{eqnarray*}
\mathbb E_o\left[\int_0^{\tau_{r}}\Gamma(X_{t})dt\right]&=&
   \int_{B_o(r)}g_r(o,x)\Gamma(x)dV(x)  \\
   &=&\int_0^rdt\int_{S_o(t)}g_r(o,x)\Gamma(x)d\sigma_t(x)  \\
&\geq& C\int_0^r\frac{\int_t^rG^{1-2m}(s)ds}{\int_1^rG^{1-2m}(s)ds}dt\int_{S_o(t)}\Gamma(x)d\sigma_t(x) \\
&=& \frac{C}{\log r}\int_0^rdt\int_t^rG^{1-2m}(s)ds\int_{S_o(t)}\Gamma(x)d\sigma_t(x)
\end{eqnarray*}
and
$$ \mathbb E_o\big{[}\Gamma(X_{\tau_r})\big{]}=\int_{S_o(r)}\Gamma(x)d\pi_o^r(x)\leq\frac{1}{\omega_{2m-1}r^{2m-1}}\int_{S_o(r)}\Gamma(x)d\sigma_r(x),$$
where
  $d\sigma_r$ is the induced volume measure  on $S_o(r).$ Thus, we have
 $$\mathbb E_o\left[\int_0^{\tau_{r}}\Gamma(X_{t})dt\right]\geq\frac{C}{\log r}\int_0^rdt\int_t^rG^{1-2m}(s)ds\int_{S_o(t)}\Gamma(x)d\sigma_t(x)$$
 and
\begin{equation}\label{fr}
  \mathbb E_o[\Gamma(X_{\tau_r})]\leq\frac{1}{\omega_{2m-1}r^{2m-1}}\int_{S_o(r)}\Gamma(x)d\sigma_r(x).
\end{equation}
 Put
$$\Gamma(r)=\int_0^rdt\int_t^rG^{1-2m}(s)ds\int_{S_o(t)}\Gamma(x)d\sigma_t(x).$$
Then
$$\Gamma(r)\leq\frac{\log r}{C}\mathbb E_o\left[\int_0^{\tau_{r}}\Gamma(X_{t})dt\right]=\hat{\Gamma}(r).
$$
Since
$$\Gamma'(r)=G^{1-2m}(r)\int_0^rdt\int_{S_o(t)}\Gamma(x)d\sigma_t(x),$$
 it yields from (\ref{fr}) that
\begin{equation}\label{B2}
  \mathbb E_o\big{[}\Gamma(X_{\tau_r})\big{]}\leq\frac{1}{\omega_{2m-1}r^{2m-1}}\frac{d}{dr}\left(\frac{\Gamma'(r)}{G^{1-2m}(r)}\right).
\end{equation}
Using Borel Lemma (Lemma \ref{cal1}) twice, then for any $\delta>0$
\begin{eqnarray*}
   && \frac{d}{dr}\left(\frac{\Gamma'(r)}{G^{1-2m}(r)}\right) \\
&\leq& G^{2m-1}(r)\Big\{\log^+\Gamma(r)\cdot\log^+\Big(G^{2m-1}(r)\Gamma(r)\big(\log^+\Gamma(r)\big)^{1+\delta}\Big)\Big\}^{1+\delta}\Gamma(r) \nonumber \\
&\leq& G^{2m-1}(r)\Big\{\log^+\hat{\Gamma}(r)\cdot\log^+\Big(G^{2m-1}(r)\hat{\Gamma}(r)\big(\log^+\hat{\Gamma}(r)\big)^{1+\delta}\Big)\Big\}^{1+\delta} \hat{\Gamma}(r) \\
&=& \frac{F(\hat{\Gamma},\kappa,\delta)G^{2m-1}(r)\log r}{C}\mathbb E_o\left[\int_0^{\tau_{r}}\Gamma(X_{t})dt\right]
\end{eqnarray*}
holds outside a set $E_\delta\subset(1,\infty)$ of finite Lebesgue measure.
By this with (\ref{B2}) and (\ref{v2}), we  have the desired inequality. Taking $``\log"$ before $F,$  the estimate is obtained.
\end{proof}

\section{A proof of Theorem \ref{thm1}}
\subsection{Preparations}~

 Let $(M,g)$ be a $m$-dimensional simply-connected   K\"ahler manifold of non-positive sectional curvature.  The associated  K\"ahler  form is defined by
 $$\alpha=\frac{\sqrt{-1}}{2\pi}\sum_{i,j=1}^mg_{i\bar{j}}dz_i\wedge d\bar{z}_j.$$
 Let $\varphi=\frac{\sqrt{-1}}{2\pi}\sum_{i,j=1}^m\varphi_{i\bar{j}}dz_i\wedge d\bar{z}_j$ be a (1,1)-form on $M.$ We use the following convenient symbols
$$\det(\varphi):=\det(\varphi_{i\bar{j}}), \ \ \ T_g(\varphi):=\sum_{i,j=1}^mg^{j\overline{i}}\varphi_{i\bar{j}},$$
where $(g^{i\bar{j}})$ is the inverse of $(g_{i\bar{j}}).$
It is trivial to show that  $T_g(\varphi)$ is  globally  defined on $M.$
\begin{lemma}\label{o1} We have
$$\varphi\wedge\alpha^{m-1}=\frac{1}{m}T_g(\varphi)\alpha^m.$$
\end{lemma}
\begin{proof} By a direct computation, it follows that
$$\frac{\varphi\wedge\alpha^{m-1}}{\alpha^m} = \frac{1}{m}\sum_{i,j=1}^m\frac{\varphi_{i\bar{j}}G_{j\bar i}}{\det(g_{s\bar{t}})},$$
where
$$G^*=\left(
        \begin{array}{ccc}
          G_{1\bar{1}} & \cdots & G_{m\bar{1}} \\
          \vdots & \ddots & \vdots \\
          G_{1\bar{m}} & \cdots & G_{m\bar{m}} \\
        \end{array}
      \right)$$
is the adjoint matrix of $G=(g_{s\bar{t}}).$ Note
$g^{i\bar{j}}=G_{i\bar{j}}/\det(g_{s\bar t}),$ hence we have
$$\sum_{i,j=1}^m\frac{\varphi_{i\bar{j}}G_{j\bar i}}{\det(g_{s\bar{t}})}
=\sum_{i,j=1}^mg^{j\bar{i}}\varphi_{i\bar{j}}=T_g(\varphi).$$
The proof is completed.
\end{proof}

\begin{lemma}\label{o2} If $\varphi$ is Hermitian semi-positive, then
$$ \left(\frac{\det(\varphi)}{\det(g_{i\overline{j}})}\right)^{\frac{1}{m}}\leq\frac{1}{m}T_g(\varphi).$$
\end{lemma}
\begin{proof} Fix $x\in M,$   take  local holomorphic coordinates $z_1\cdots,z_m$ around $x$ such that $g_{i\bar{j}}(x)=\delta_{ij}.$ At $x$, the inequality is equivalent to
$$\left(\det(\varphi)\right)^{\frac{1}{m}}\leq\frac{1}{m}{\rm{tr}}(\varphi)$$
which  holds clearly. In fact, the linear algebra theory asserts that
$$\det(\varphi)=\lambda_1\cdots\lambda_m, \ \ \ {\rm{tr}}(\varphi)=\lambda_1+\cdots+\lambda_m,$$
where $\lambda_1,\cdots,\lambda_m$ are  eigenvalues of $(\varphi_{i\overline{j}}).$
Since $\varphi$ is Hermitian semi-positive, then  $\lambda_1,\cdots,\lambda_m\geq 0.$  
The mean-value inequality implies the lemma.
\end{proof}
\noindent\textbf{Wong-Lang's approach}

Let $V$ be a complex projective manifold and let $L\rightarrow V$ be a positive line bundle.
Let a reduced divisor
$D\in |L|$ be of simple normal crossing type, we write $D=\sum_{j=1}^q D_j$ as the union of irreducible components, i.e., $D_1,\cdots,D_q$ are
irreducible and non-singular, moreover, at each point $x$ of $V$ there exists a local holomorphic coordinate neighborhood  $U(z_1,\cdots,z_m)$ of $x$ such that
$$D\cap U=\big\{z_1\cdots z_{k_x}=0\big\}, \ \ \ 0\leq k_x\leq m.$$
If $k_x=0,$ then it means that $D\cap U=\emptyset.$ Set
\begin{equation}\label{com}
k:=\max_{x\in V} k_x,
\end{equation}
which is called the \emph{complexity} of $D.$
Denote by $s_j(1\leq j\leq q)$ the canonical section defined by $D_j.$ Clearly, $s_D=s_1\otimes\cdots\otimes s_q$ gives the canonical section defined by $D.$ Endowing $L_{D_j}(1\leq j\leq q)$ with a Hermitian metric $h_j,$ which induces a natural Hermitian metric $h$ on $L.$ Since $L>0,$ one may assume that $c_1(L,h)=-dd^c\log h>0.$
Define the singular volume form
\begin{equation}\label{sing}
  \Phi_{D,\lambda}:=\frac{\Omega}{\prod_{j=1}^q\|s_j\|^{2(1-\lambda)}}, \ \ \ \Omega=(-dd^c\log h)^m
\end{equation}
on  $V,$ where  $0<\lambda<1$ is a constant.
Set
$$\eta_{D,\lambda}:=(1+q)\lambda c_1(L,h)+\sum_{j=1}^qdd^c\log(1+\|s_j\|^{2\lambda}).$$
\ \ \ \ Lang proved that
\begin{lemma}[Lemma 7.4, \cite{Lang}]\label{op} There exists a number $b>0$ depending only on $D,$ $\Omega$ and $c_1(L,h)$ such that
  $$\lambda^{m+k}\Phi_{D,\lambda}\leq b\eta_{D,\lambda}^m.$$
\end{lemma}
\noindent\textbf{Estimate of $\log F(\hat{\Gamma},\kappa,\delta)$ with  $\Gamma=\xi^{1/m}$}

Let  $f:M\rightarrow V$ be a differentiably non-degenerate equi-dimensional holomorphic mapping, i.e., the differential $df$ has rank $m$ at a  point of $M.$
Write
\begin{equation}\label{defii1}
  \Omega=a(w)\bigwedge_{j=1}^m\frac{\sqrt{-1}}{2\pi}dw_j\wedge d\bar{w}_j
\end{equation}
in a local holomorphic  coordinate system $w.$  It follows that
$$f^*\Omega=a(f)|J(f)|^2\bigwedge_{j=1}^m\frac{\sqrt{-1}}{2\pi}dz_j\wedge d\bar{z}_j$$
in a local holomorphic  coordinate system $z$ of $M,$ where $J(f)$ is the Jacobian determinant of $f.$ Clearly, the zero divisor $(J(f))$ is globally  defined. In fact, for $x\in M$, given two local holomorphic  coordinate systems $z,\tilde{z}$ near $x$ and  two  local holomorphic  coordinate systems $w,\tilde{w}$ near $f(x),$ then 
\begin{eqnarray*}
J(f(z))&=&\left|\frac{\partial(w_1,\cdots,w_m)}{\partial(z_1,\cdots,z_m)}\right| \\
&=&\left|\frac{\partial(w_1,\cdots,w_m)}{\partial(\tilde{w}_1,\cdots,\tilde{w}_m)}\right|
\left|\frac{\partial(\tilde{w}_1,\cdots,\tilde{w}_m)}{\partial(\tilde{z}_1,\cdots,\tilde{z}_m)}\right|
\left|\frac{\partial(\tilde{z}_1,\cdots,\tilde{z}_m)}{\partial(z_1,\cdots,z_m)}\right| \\
&=&\left|\frac{\partial(w_1,\cdots,w_m)}{\partial(\tilde{w}_1,\cdots,\tilde{w}_m)}\right|J(f(\tilde{z}))
\left|\frac{\partial(\tilde{z}_1,\cdots,\tilde{z}_m)}{\partial(z_1,\cdots,z_m)}\right|.
\end{eqnarray*}
 We use ${\rm{Ram}}_f$ to denote  $(J(f)),$ called the ramification  divisor  of $f.$

\begin{lemma}\label{}
  Set
$f^*\Phi_{D,\lambda}=\xi\alpha^m,$
 where $\Phi_{D,\lambda}$ is defined by $(\ref{sing})$. Then
$$\xi^{\frac{1}{m}}\leq \frac{(q+1)b^{\frac{1}{m}}}{2m\lambda^{\frac{k}{m}}}e_{f^*c_1(L,h)}+\frac{b^{\frac{1}{m}}}{4m\lambda^{1+\frac{k}{m}}}\sum_{j=1}^q\Delta_M
\log\big{(}1+\|s_j\circ f\|^{2\lambda}\big{)}.$$
\end{lemma}
\begin{proof} By Lemma \ref{o1}-Lemma \ref{op}, we directly compute that
\begin{eqnarray*}
\xi^{\frac{1}{m}}
&\leq& \frac{b^{\frac{1}{m}}}{\lambda^{1+\frac{k}{m}}}\left(\frac{f^*\eta^m_{D,\lambda}}{\alpha^m}\right)^{1\over m} \\
&=& \frac{b^{\frac{1}{m}}}{\lambda^{1+\frac{k}{m}}}\left(\frac{\det(f^*\eta_{D,\lambda})}{\det(g_{i\bar{j}})}\right)^{1\over m} \\
&\leq& \frac{b^{\frac{1}{m}}}{m\lambda^{1+\frac{k}{m}}}T_g(f^*\eta_{D,\lambda}) 
=  \frac{b^{\frac{1}{m}}}{\lambda^{1+\frac{k}{m}}}\frac{f^*\eta_{D,\lambda}\wedge\alpha^{m-1}}{\alpha^m} \\
&=& \frac{b^{\frac{1}{m}}}{\lambda^{1+\frac{k}{m}}}\frac{\left((q+1)\lambda f^*c_1(L,h)+\sum_{j=1}^qdd^c\log\big{(}1+\|s_j\circ f\|^{2\lambda}\big{)}\right)\wedge\alpha^{m-1}}{\alpha^m} \\
&=& \frac{(q+1)b^{\frac{1}{m}}}{2m\lambda^{\frac{k}{m}}}e_{f^*c_1(L,h)}+\frac{b^{\frac{1}{m}}}{4m\lambda^{1+\frac{k}{m}}}\sum_{j=1}^q\Delta_M
\log\big{(}1+\|s_j\circ f\|^{2\lambda}\big{)},
\end{eqnarray*}
where $b>0$ is a suitable number independent of $\lambda.$
\end{proof}
\begin{lemma}\label{oo} There exists a  number $b>0$ independent of $\lambda$ such that
  \begin{eqnarray*}
  \mathbb E_o\left[\int_0^{\tau_r}\xi^{\frac{1}{m}}(X_t)dt\right]
  &\leq& \frac{b^{\frac{1}{m}}}{m\lambda^{\frac{k}{m}}}\Big((q+1)T_f(r,L)+\frac{q\log2}{2\lambda}\Big)
\end{eqnarray*}
holds for any constant $0<\lambda<1.$ Moreover, $\lambda$ can be replaced by a function $\kappa$ satisfying $0<\kappa(r)\leq c_0<1.$ If take  $\kappa(r)=1/T_f(r,L),$
then there exists a number $B>0$ such that
$$\mathbb E_o\left[\int_0^{\tau_r}\xi^{\frac{1}{m}}(X_t)dt\right]\leq BT^{1+{k\over m}}_f(r,L)$$
for $r>0$ large enough, where
$$B\geq \Big(1+q+\frac{q\log2}{2}\Big)\frac{b^{\frac{1}{m}}}{m}.$$
\end{lemma}
\begin{proof} Using Lemma \ref{oo}, we obtain
\begin{eqnarray*}
  \mathbb E_o\left[\int_0^{\tau_r}\xi^{\frac{1}{m}}(X_t)dt\right]
  &\leq& \frac{(q+1)b^{\frac{1}{m}}}{2m\lambda^{\frac{k}{m}}}\mathbb E_o\left[\int_0^{\tau_r}e_{f^*c_1(L,h)}(X_t)dt\right] \\
&&+\frac{b^{\frac{1}{m}}}{4m\lambda^{1+\frac{k}{m}}}\sum_{j=1}^q\mathbb E_o\left[\int_0^{\tau_r}\Delta_M
\log\big{(}1+\|s_j\circ f\|^{2\lambda}\big{)}(X_t)dt\right],
\end{eqnarray*}
where $b$ is independent of $\lambda$ with $0<\lambda<1.$ Observing that
$$\mathbb E_o\left[\int_0^{\tau_r}e_{f^*c_1(L,h)}(X_t)dt\right]=2T_f(r,L)$$
and  Dynkin formula implies
\begin{eqnarray*}
&&\mathbb E_o\left[\int_0^{\tau_r}\Delta_M
\log\big{(}1+\|s_j\circ f(X_t)\|^{2\lambda}\big{)}dt\right]\\
&=& 2\mathbb E_o\left[\log\big{(}1+\|s_j\circ f(X_{\tau_r})\|^{2\lambda}\big{)}\right]-
2\log\big{(}1+\|s_j\circ f(o)\|^{2\lambda}\big{)} \\
&<& 2\log2
\end{eqnarray*}
since the assumption that $\|s_j\|<1$ for $1\leq j\leq q.$
Thus, we conclude that
\begin{eqnarray*}
  \mathbb E_o\left[\int_0^{\tau_r}\xi^{\frac{1}{m}}(X_t)dt\right]
  &\leq& \frac{b^{\frac{1}{m}}}{m\lambda^{\frac{k}{m}}}\Big((q+1)T_f(r,L)+\frac{q\log2}{2\lambda}\Big).
\end{eqnarray*}
The independence of $b$ from $\lambda$ implies that  $\lambda$ could  be replaced by a function $\kappa$ satisfying $0<\kappa(r)\leq c_0<1.$ Since $f$ is non-degenerate, then we have that $T_f(r,L)>1$ when $r$ is large enough. Replacing $\lambda$ by $1/T_f(r,L),$ we conclude that
$$\mathbb E_o\left[\int_0^{\tau_r}\xi^{\frac{1}{m}}(X_t)dt\right]\leq \left(1+q+\frac{q}{2}\log2\right)\frac{b^{\frac{1}{m}}}{m}T^{1+{k\over m}}_f(r,L)$$
for $r>1$ large enough. The proof is completed.
\end{proof}

\begin{lemma}\label{last} Set $\Gamma=\xi^{\frac{1}{m}},$ we have
$$\log F(\hat{\Gamma},\kappa,\delta)
\leq
 O\Big(\log^+\log T_f(r,L)+\log^+\big(r\sqrt{-\kappa(r)}\big)+\log^+\log r\Big). $$
holds for $r>1$ large enough, where $F$ is defined in Lemma $\ref{cal}.$
\end{lemma}
\begin{proof}  Lemma $\ref{cal}$ implies that
\begin{eqnarray}\label{x1}
 &&\log F(\hat{\Gamma},\kappa,\delta) \\
&\leq& O\Big(\log^+\log \mathbb E_o\left[\int_0^{\tau_{r}}\xi^{\frac{1}{m}}(X_{t})dt\right]+\log^+\big(r\sqrt{-\kappa(r)}\big)+\log^+\log r\Big). \nonumber
\end{eqnarray}
Note by Lemma \ref{oo} that there exists a number $B>0$ such that
\begin{equation}\label{x2}
  \mathbb E_o\left[\int_0^{\tau_r}\xi^{\frac{1}{m}}(X_t)dt\right]\leq BT^{1+{k\over m}}_f(r,L)
\end{equation}
for $r>1$ large enough. Combining (\ref{x1}) with (\ref{x2}), we prove the lemma.
\end{proof}
\noindent\textbf{Estimate of  $T(r,\mathscr{R}_M)$}

Write
${\rm{Ric}}_M=\sum_{i,j}R_{i\bar{j}}dz_i\otimes d\bar{z}_j,$  where
\begin{equation}\label{syy}
  R_{i\bar{j}}=-\frac{\partial^2}{\partial z_i\partial \bar{z}_j}\log\det(g_{s\bar{t}}).
\end{equation}
Let $s_M$ be the scalar curvature of $M$ defined by
$$s_M=\sum_{i,j=1}^mg^{i\bar j}R_{i\bar j},$$
where
$(g^{i\bar j})$ is the inverse of $(g_{i\bar j}).$ By virtue of (\ref{syy}), we obtain
$$
  s_M=-\frac{1}{4}\Delta_M\log\det(g_{s\bar t}).
$$
\begin{lemma}\label{s123}  We have
$$s_M\geq m R_M.$$
\end{lemma}
\begin{proof} Fix any point $x\in M$, we  take a local holomorphic coordinate system $z$ near $x$ such that $g_{i\bar{j}}(x)=\delta^i_j.$
Then 
\begin{eqnarray*}
s_M(x)&=&\sum_{j=1}^mR_{j\bar{j}}(x) =\sum_{j=1}^m{\rm{Ric}}_M(\frac{\partial}{\partial z_j},\frac{\partial}{\partial \bar{z}_j})_x\geq mR_M(x)
\end{eqnarray*}
which proves the lemma.
\end{proof}

\begin{lemma}\label{yyy} We have
$$\mathbb E_o\big{[}\tau_r\big{]}\leq\frac{2r^2}{2m-1}.$$
\end{lemma}
\begin{proof}   
Let $X_t$ be the Brownian motion in $M$ started at $o\not=o_1,$ where $o_1\in B_o(r).$  Let  $r_1(x)$ be  the distance function  of $x$ from $o_1.$
Apply It$\rm{\hat{o}}$ formula to $r_1(x)$
\begin{equation}\label{kiss}
  r_1(X_t)-r_1(X_0)=B_t-L_t+\frac{1}{2}\int_0^t\Delta_Mr_1(X_s)ds,
\end{equation}
here $B_t$ is the standard Brownian motion in $\mathbb R,$ and $L_t$ is a local time on cut locus of $o,$ an increasing process
which increases only at cut loci of $o.$ Since $M$ is simply connected and  non-positively  curved, then
$$\Delta_Mr_1(x)\geq\frac{2m-1}{r_1(x)}, \ \ L_t\equiv0.$$
By (\ref{kiss}), we arrive at
$$r_1(X_t)\geq B_t+\frac{2m-1}{2}\int_0^t\frac{ds}{r_1(X_s)}.$$
Let $t=\tau_r$ and take expectation on both sides of the above inequality, then it yields that 
$$\max_{x\in S_o(r)} r_1(x)\geq \frac{(2m-1)\mathbb E_o[\tau_r]}{2\max_{x\in S_o(r)} r_1(x)}.$$
Let $o'\rightarrow o,$ 
we are led to the conclusion.
\end{proof}
\begin{lemma}\label{b0}
Let $\kappa$ be defined by $(\ref{kappa}).$ We have
$$T(r,\mathscr{R}_M)\geq2m\kappa(r)r^2.$$
\end{lemma}
\begin{proof} Lemma \ref{s123} implies that
$0\geq s_M\geq mR_M.$
By co-area formula
\begin{eqnarray*}
T(r,\mathscr{R}_M)&=&-{1\over4}\mathbb E_o\left[\int_0^{\tau_r}\Delta_M\log\det(g_{i\bar{j}}(X_t))dt\right] \\
&=&\mathbb E_o\left[\int_0^{\tau_r}s_{M}(X_t)dt\right] \geq m\mathbb E_o\left[\int_0^{\tau_r}R_M(X_t)dt\right] \\
&\geq& m(2m-1)\kappa(r)\mathbb E_o[\tau_r].
\end{eqnarray*}
Using Lemma \ref{yyy}, we show the lemma.
\end{proof}
\subsection{Proof of Theorem \ref{thm1}}~

Consider  the (analytic) universal covering $$\pi:\tilde{M}\rightarrow M.$$ Via the pull-back of $\pi,$ $\tilde{M}$ can be equipped with the induced metric
 from the
metric of $M.$ So, under this metric, $\tilde{M}$ becomes a simply-connected complete K\"ahler manifold of non-positive sectional curvature.
Take a diffusion process $\tilde{X}_t$ in $\tilde{M}$  such that $X_t=\pi(\tilde{X}_t),$ where $X_t$ is the Brownian motion started at $o\in M.$ Then
$\tilde{X}_t$ is the Brownian motion generated by $\Delta_{\tilde{M}}/2$ induced from the pull-back metric. Let $\tilde{X}_t$ start at $\tilde{o}\in\tilde{M}$  with $o=\pi(\tilde{o}),$  we have
$$\mathbb E_o[\phi(X_t)]=\mathbb E_{\tilde{o}}\big{[}\phi\circ\pi(\tilde{X}_t)\big{]}$$
for $\phi\in \mathscr{C}_{\flat}(M).$ Set $$\tilde{\tau}_r=\inf\big{\{}t>0: \tilde{X}_t\not\in B_{\tilde{o}}(r)\big{\}},$$ where
$B_{\tilde{o}}(r)$ is a geodesic ball centered at $\tilde{o}$ with radius $r$ in $\tilde{M}.$
 If necessary, one can extend the filtration in probability space where $(X_t,\mathbb P_o)$ are defined so that $\tilde{\tau}_r$ is a stopping time with
 respect to a filtration where the stochastic calculus of $X_t$ works.
By the above arguments, we may assume $M$ is simply connected without loss of generality by lifting $f$ to the covering, see \cite{atsuji}.

\begin{proof}
The equality
$$f^*\Phi_{D,\lambda}=\xi\alpha^m,$$
 where $\Phi_{D,\lambda}$ is defined by (\ref{sing}), implies  that
\begin{eqnarray*}
   dd^c\log\xi&=&(1-\lambda)f^*c_1(L,h)-f^*{\rm{Ric}}\Omega+\mathscr R_M-(1-\lambda)f^*D+{\rm{Ram}}_f \\
   &\geq& (1-\lambda)f^*c_1(L,h)-f^*{\rm{Ric}}\Omega+\mathscr R_M-{\rm{Supp}}f^*D
\end{eqnarray*}
in the sense of currents. Thus,
\begin{eqnarray}\label{5q}
&&\frac{1}{4}\int_{B_o(r)}g_r(o,x)\Delta_M\log\xi(x) dV(x) \\
&\geq& (1-\lambda)T_f(r,L)+T_f(r,K_V)+T(r,\mathscr{R}_M)-N^{[1]}_f(r,D)+O(1), \nonumber
\end{eqnarray}
where $K_V$ is the canonical line bundle over $V$. By Dynkin formula
\begin{eqnarray*}
\frac{1}{2}\mathbb E_o\left[\int_0^{\tau_r}\Delta_M\log\xi(X_t)dt\right]
=\mathbb E_o\big{[}\log\xi(X_{\tau_r})\big{]}-\log\xi(o).
\end{eqnarray*}
By this with (\ref{5q}) to get
\begin{eqnarray*}
 && \frac{1}{2}\mathbb E_o\big{[}\log\xi(X_{\tau_r})\big{]} \\
 &\geq& (1-\lambda)T_f(r,L)+T_f(r,K_V)+T(r,\mathscr{R}_M)-N^{[1]}_f(r,D) +O(1). \nonumber
\end{eqnarray*}
Take $r_0>0$ such that $T_f(r,L)>1$ as $r\geq r_0.$ Replacing $\lambda$ by $1/T_f(r,L),$ we obtain
\begin{eqnarray}\label{w1}
 && \frac{1}{2}\mathbb E_o\big{[}\log\xi(X_{\tau_r})\big{]}  \\
 &\geq& T_f(r,L)+T_f(r,K_V)+T(r,\mathscr{R}_M)  -N^{[1]}_f(r,D) +O(1) \nonumber
\end{eqnarray}
for $r>1$ large enough.
On the other hand, using Lemma \ref{cal}, for any $\delta>0,$ there exists a set $E'_{\delta}\subset(1,\infty)$ such that
\begin{eqnarray*}
  && \frac{1}{2}\mathbb E_o\big{[}\log\xi(X_{\tau_r})\big{]} \\
   &\leq& \frac{m}{2}\log\mathbb E_o\big{[}\xi^{\frac{1}{m}}(X_{\tau_r})\big{]} \\
&\leq& \frac{m}{2}\log \mathbb E_o\left[\int_0^{\tau_{r}}\xi^{1\over m}(X_{t})dt\right]+
 \frac{m}{2}\log\frac{F(\hat{\Gamma},\kappa,\delta)e^{(2m-1)r\sqrt{-\kappa(r)}}\log r}{C\omega_{2m-1}}\\
&:=& \frac{m}{2}(A_1+A_2)
\end{eqnarray*}
holds for $r>1$ outside $E'_{\delta}$ with $\Gamma=\xi^{1/m}.$
For $A_1,$  apply Lemma \ref{oo} to get
$$A_1\leq \frac{m+k}{m}\log T_f(r,L)+O(1)$$
 as $r>r_0.$  For $A_2,$ by  Lemma \ref{last} we have
\begin{eqnarray*}
A_2&\leq& \log F(\hat{\Gamma},\kappa,\delta)+(2m-1)r\sqrt{-\kappa(r)}+\log^+\log r+O(1) \\
&\leq&  O\Big(\log^+\log T_f(r,L)+\log^+\big(r\sqrt{-\kappa(r)}\big)+\log^+\log r\Big) \\
&&+(2m-1)r\sqrt{-\kappa(r)}+\log^+\log r+O(1) \\
&\leq&  O\Big(\log^+\log T_f(r,L)+r\sqrt{-\kappa(r)}+\log^+\log r\Big)
\end{eqnarray*}
as $r>r_0.$ Hence, it finally follows that
\begin{eqnarray*}\label{w2}
  && \frac{1}{2}\mathbb E_o\big{[}\log\xi(X_{\tau_r})\big{]} \\
&\leq& \frac{m+k}{2}\log T_f(r,L)+ O\Big(\log^+\log T_f(r,L)+r\sqrt{-\kappa(r)}+\log^+\log r\Big)
\end{eqnarray*}
for $r>1$ outside $E_\delta=E'_\delta\cup (0,r_0]$. By the above with (\ref{w1}) and Lemma \ref{b0}, the theorem is proved.
\end{proof}

\vskip\baselineskip

\label{lastpage-01}
\end{document}